\documentclass[a4paper,11pt]{article}
\usepackage[english]{babel}
\usepackage[latin1]{inputenc}
\usepackage{amssymb}
\usepackage{amscd}
\usepackage{mathtools}
\usepackage[paper=a4paper,left=3cm,right=3cm,top=3cm,bottom=2.3cm]{geometry}
\usepackage[thmmarks,amsmath,amsthm]{ntheorem}

\usepackage{hyperref}

\newtheorem{thm}{Theorem}[section]
\newtheorem{lem}[thm]{Lemma}
\newtheorem{defi}[thm]{Definition}
\newtheorem{prop}[thm]{Proposition}

\newtheorem{rem}[thm]{Remark}
\newtheorem{exam}[thm]{Example}

\begin{document}

\title{A Deligne pairing for Hermitian Azumaya modules}
\author{Fabian Reede\footnote{Georg-August-Universit\"at G\"ottingen, freede@uni-math.gwdg.de}}
\maketitle

\begin{abstract}
In this short note we want to give a definition of a generalized Deligne pairing for modules over an Azumaya algebra on an arithmetic surface $X$. We do this by defining  Hermitian metrics on the Azumaya algebra and on the modules in question. Then we go on and define the determinant of the cohomology for a pair of modules over an Azumaya algebra. Using this we give the definition of a generalized Deligne pairing and study some of its properties. 
\end{abstract}
\section*{Introduction}
Assume we are given an arithmetic surface $\pi: X\rightarrow Y$ and two line bundles $\mathcal{L}$ and $\mathcal{M}$ on $X$. In (\cite{deli}) Deligne associated to this data a line bundle $\left\langle \mathcal{L},\mathcal{M}\right\rangle$ on $Y$. This line bundle is nowadays called the Deligne pairing of the two line bundles. If these line bundles are in fact Hermitian line bundles, then Deligne showed how to construct a Hermitian metric on the pairing, using the two given metrics. That is, the Deligne pairing gives a Hermitian line bundle on $Y$. This Hermitian line bundle has some interesting properties if one looks at it in Arakelov Geometry, see for example (\cite[Theorem 4.10.1.]{soul2}) or (\cite[1.3.]{soul3}).\\
In this note we generalize this pairing to modules over an Azumaya algebra on an arithmetic surface $X$.
\medskip

\noindent In the first section we recall the basic definitions of Azumaya algebras on schemes and of Brauer groups. We go on and recall some facts from Morita theory which we will need. Then we can define Hermitian metrics on Azumaya algebras and their modules.
\medskip

\noindent In the second section we define the determinant of the cohomology of a pair of Azumaya modules. Using this we give the definition of the Deligne pairing for Azumaya modules and prove that for the trivial Azumaya algebra we get the well known Deligne pairing for Hermitian line bundles back. 
\subsection*{Notation}
All Hermitian metrics in this note are supposed to be invariant with respect to the complex conjugation on $\mathbb{C}$, see (\cite[Definition IV.4.1.4.]{soul}).\\
All sheaves in this note are supposed to be coherent if not otherwise stated.\\
In this note we always assume that if $\mathcal{A}$ is an Azumaya algebra on a scheme $X$, then $\mathcal{A}_\eta$ is a division ring over the function field $k(X)=\mathcal{O}_{X,\eta}$, here $\eta \in X$ is the generic point of the scheme $X$.
\tableofcontents

\section{Hermitian Azumaya modules}

\subsection{Azumaya algebras and Azumaya modules}
\begin{defi}
If $X$ is any scheme, then a locally free sheaf of locally finite rank equipped with the structure of a sheaf of $\mathcal{O}_X$-algebras is called an Azumaya algebra over $X$ if the following condition is satisfied:\\ 
For every closed point $x\in X$ the $k(x)$-algebra $\mathcal{A}(x):=\mathcal{A}\otimes_{\mathcal{O}_X} k(x)$ is a central simple algebra over the residue field $k(x)$.
\end{defi}
\begin{prop}\label{azcrit}
If $X$ is any scheme, then a locally free sheaf of locally finite rank equipped with the structure of a sheaf of $\mathcal{O}_X$-algebras is an Azumaya algebra over $X$ if and only if the canonical homomorphism
\begin{alignat*}{3}
\mathcal{A}&\otimes_{\mathcal{O}_X}\mathcal{A}^{op} &&\longrightarrow{}{} &&\mathcal{E}nd_{\mathcal{O}_X}(\mathcal{A})\\
a&\otimes a' &&\longmapsto &&(x\mapsto axa')
\end{alignat*} 
is an isomorphism.
\end{prop}
\begin{rem}
\normalfont Here $\mathcal{A}^{op}$ denotes the opposite ring. It has the same underlying abelian group as $\mathcal{A}$ but the multiplication is reversed.
\end{rem}
We will also need the following lemma which describes the behaviour of Azumaya algebras with respect to pullbacks.
\begin{lem}\label{pullaz}
Let $f:X \rightarrow Y$ be a morphism of schemes. If $\mathcal{A}$ is an Azumaya algebra over $Y$ then $f^{*}\mathcal{A}$ is an Azumaya algebra over $X$. 
\end{lem}
\begin{defi}
Let $X$ be any scheme. Two Azumaya algebras $\mathcal{A}$ and $\mathcal{A}'$ are similar, $\mathcal{A}\sim\mathcal{A}'$, if there exist two locally free $\mathcal{O}_X$-modules $\mathcal{E}$ and $\mathcal{F}$ such that there is an isomorphism
\begin{center}
$\mathcal{A}\otimes_{\mathcal{O}_X}\mathcal{E}nd_{\mathcal{O}_X}(\mathcal{E}) \cong \mathcal{A}'\otimes_{\mathcal{O}_X}\mathcal{E}nd_{\mathcal{O}_X}(\mathcal{F})$.
\end{center}
\end{defi}
It is an easy exercise to see that similarity is in fact an equivalence relation. Using this fact we can make the following definition:
\begin{defi}
Let $X$ be any scheme. The set of all similarity classes of Azumaya algebras over $X$ is called the Brauer group of $X$. It is denoted by $Br(X)$.
\end{defi}
\begin{rem}
\normalfont The set $Br(X)$ is in fact a group:\\ 
If $[\mathcal{A}]$ denotes the similarity class of $\mathcal{A}$, then $Br(X)$ turns into a group using the multiplication 
\begin{center}
$[\mathcal{A}]\bullet[\mathcal{A}']:=[\mathcal{A}\otimes_{\mathcal{O}_X}\mathcal{A}']$. 
\end{center}
The identity element is given by $[\mathcal{O}_X]$ and the inverse of $[\mathcal{A}]$ is given by $[\mathcal{A}^{op}]$ because of (\ref{azcrit}). 
\end{rem}
\begin{rem}
\normalfont If $X=Spec(K)$ for a field $K$, then $Br(X)$ equals the usual Brauer group $Br(K)$ of Azumaya algebras over $K$. Especially we have $Br(\mathbb{C})=0$, $Br(\mathbb{F}_q)=0$ and $Br(\mathbb{R})=\mathbb{Z}/2\mathbb{Z}$. 
\end{rem}
\begin{exam}\label{brauer}
\normalfont Let $X$ be a smooth proper curve over an algebraically closed field $k$. Then one has $Br(k(X))=0$. This fact is known as Tsen's theorem, (\cite[Th\'{e}or\`{e}me (1.1)]{groth1}). As a consequence we get $Br(X)=0$, since we have an injection $Br(X)\hookrightarrow Br(k(X))$ in this case, see for example (\cite[Corollaire (1.2)]{groth1}).
\end{exam}
We will also need the notion of modules over an Azumaya algebra.
\begin{defi}
Let $X$ be any scheme and let $\mathcal{A}$ be an Azumaya algebra over $X$. A sheaf of $\mathcal{O}_X$-modules $\mathcal{M}$ is called an Azumaya module if it has the structure of a locally projective left $\mathcal{A}$-module. An Azumaya module is called an Azumaya line bundle, if $\mathcal{M}_{\eta}$ is a one-dimensional vector space over the division ring $\mathcal{A}_{\eta}$.
\end{defi}
\begin{rem}
\normalfont Since $\mathcal{A}$ is a locally free $\mathcal{O}_X$-module, we see that any Azumaya module is also locally free as an $\mathcal{O}_X$-module.
\end{rem}
\begin{rem}
\normalfont We note that one can also define more general Azumaya modules, for example by allowing torsion free modules instead of locally projective ones. Since we are only interested in locally projective modules in the following we will use this as our definition of Azumaya modules.
\end{rem}

\subsection{Morita equivalence}
If $C$ is a smooth proper curve over an algebraically closed field, that is, a smooth projective curve, then we have seen in $(\ref{brauer})$ that $Br(C)=0$. So if $\mathcal{A}$ is an Azumaya algebra over $C$, then there exists a locally free $\mathcal{O}_C$-module $\mathcal{E}$ and an isomorphism 
\begin{center}
$\mathcal{A}\cong \mathcal{E}nd_{\mathcal{O}_C}(\mathcal{E})$.
\end{center}
In general if an Azumaya algebra $\mathcal{A}$ on a scheme $X$ is of the form $\mathcal{E}nd_{\mathcal{O}_X}(\mathcal{E})$, for a locally free sheaf $\mathcal{E}$ on $X$, then we can use Morita equivalence to describe the category of $\mathcal{A}$-modules.
\medskip

\noindent We will now recall a few facts from Morita theory, which we will need in the following sections. To begin we note that $\mathcal{E}$ is naturally a left $\mathcal{A}$-module and the module structure is given by evaluation. Furthermore $\mathcal{E}^{*}:=\mathcal{H}om_{\mathcal{O}_X}(\mathcal{E},\mathcal{O}_X)$ is naturally a right $\mathcal{A}$-module and the module structure is given by composition.
\begin{lem}
Let $X$ be a scheme. If $\mathcal{A}\cong \mathcal{E}nd_{\mathcal{O}_X}(\mathcal{E})$ is an Azumaya algebra over $X$, then the evaluation map 
\begin{center}
$\mathcal{E}^{*}\otimes_{\mathcal{A}}\mathcal{E}\rightarrow \mathcal{O}_X$ 
\end{center}
is an isomorphism of $\mathcal{O}_X$-modules.
\end{lem}
\begin{proof}
Since the evaluation map defines a sheaf map $\mathcal{E}^{*}\otimes_{\mathcal{A}}\mathcal{E}\rightarrow \mathcal{O}_X$ is is enough to check that this is an isomorphism on the stalks.\\ 
So we pick a point $x\in X$. We get a commutative ring $R=\mathcal{O}_{X,x}$ and a free $R$-module $M=\mathcal{E}_x$. We want to prove that the evaluation map 
\begin{center}
$M^{*}\otimes_{End_R(M)}M\rightarrow R$
\end{center}
is an isomorphism. To do this we pick a basis $m_1,\ldots,m_r$ of $M$, which induces a dual basis $m_1^{*},\ldots,m_r^{*}$ of $M^{*}$.\\ 
It is a short computation to see that we have:
\begin{itemize}
\item $m_i^{*}\otimes m_i=m_j^{*}\otimes m_j$ for all $i,j$. 
\item $m_i^{*}\otimes m_j=0$ for $i\neq j$ 
\end{itemize}
Now every element $z$ in $M^{*}\otimes_{End_R(M)}M$ is a sum of pure tensors, that is, we can write $z=\sum a_{ij} m_i^{*}\otimes m_j$. The results show that z simplifies to $z=(\sum a_{ii})m_1^{*}\otimes m_1$. So $z$ maps to $\sum a_{ii}$ under the evaluation map. But this map is an isomorphism.  
\end{proof}
\begin{lem}
Let $X$ be a scheme. If $\mathcal{A}\cong \mathcal{E}nd_{\mathcal{O}_X}(\mathcal{E})$ is an Azumaya algebra over $X$, then for any two $\mathcal{O}_X$-modules $\mathcal{M}$ and $\mathcal{N}$ there is an isomorphism
\begin{center}
$\mathcal{H}om_{\mathcal{O}_X}(\mathcal{M},\mathcal{N}) \cong \mathcal{H}om_{\mathcal{A}}(\mathcal{E}\otimes_{\mathcal{O}_X}\mathcal{M},\mathcal{E}\otimes_{\mathcal{O}_X}\mathcal{N})$. 
\end{center}
\end{lem}
\begin{proof}
The map in one direction is given by $id_{\mathcal{E}}\otimes_{\mathcal{O}_X}(\cdot)$. The map in the other direction is given by $id_{\mathcal{E}^{*}}\otimes_{\mathcal{A}}(\cdot)$ and by using the previous lemma.
\end{proof}
Using theses two results one can prove the desired Morita equivalence in this situation. We use the following notation: if $X$ is a scheme and $\mathcal{A}$ is an Azumaya algebra over $X$, then the category of left $\mathcal{A}$-modules is denoted by $Mod(\mathcal{A})$. The category of $\mathcal{O}_X$-modules is denoted by $Mod(\mathcal{O}_X)$. 
\begin{prop}\label{morita}
Let $X$ be a scheme and let $\mathcal{A}$ be an Azumaya algebra over $X$. If $\mathcal{A}\cong \mathcal{E}nd_{\mathcal{O}_X}(\mathcal{E})$ for a locally free sheaf $\mathcal{E}$ on $X$, then the functor
\begin{alignat*}{3}
F:{} &Mod(\mathcal{O}_X) &&\longrightarrow{} &&Mod(\mathcal{A})\\
&\hspace{0.5cm}\mathcal{M} &&\longmapsto{} &&\mathcal{E}\otimes_{\mathcal{O}_X}\mathcal{M}
\end{alignat*} 
is an equivalence of categories.
\end{prop}
\begin{rem}
\normalfont This equivalence maps locally free $\mathcal{O}_X$-modules to locally projective $\mathcal{A}$-modules since $\mathcal{E}$ is a locally projective $\mathcal{A}$-module. 
\end{rem}

\subsection{Hermitian Azumaya Modules}
\begin{defi}
Let $Y$ be a Dedekind scheme. A two-dimensional integral regular projective flat $Y$-scheme $X$ is called an arithmetic surface and we denote the structure morphism by $\pi: X  \rightarrow Y$.
\end{defi}
For simplicity we will only study the case where $Y$ is the spectrum of the integers in the following. So from now on we always have $Y=Spec(\mathbb{Z})$.
\medskip

\noindent Since we have $k(Y)=\mathbb{Q}$, we denote the generic fiber of $\pi$ by $X_{\mathbb{Q}}$. It is known that $X_{\mathbb{Q}}$ is normal, see (\cite[Lemma 8.3.3]{liu}). As $dim(X_{\mathbb{Q}})=1$ it follows that $X_{\mathbb{Q}}$ is also regular and so the generic fiber is in fact smooth over $\mathbb{Q}$, since $\mathbb{Q}$ is a perfect field, see (\cite[Corollary 4.3.33]{liu}). Using the embedding $\mathbb{Q}\hookrightarrow \mathbb{C}$ we can change the base and get a smooth projective curve $X_{\mathbb{C}}$ over the algebraically closed field $\mathbb{C}$. If $\mathcal{F}$ is a sheaf of $\mathcal{O}_X$-modules, then we denote the induced sheaves on $X_{\mathbb{Q}}$ respectively $X_{\mathbb{C}}$ by $\mathcal{F}_{\mathbb{Q}}$ and $\mathcal{F}_{\mathbb{C}}$.
\medskip

\noindent The associated Riemann surface $S$ of $X$ is given by $S=X_{\mathbb{C}}(\mathbb{C})$ and comes endowed with a Hermitian metric $\omega$, which is K\"{a}hler in this case. If $\mathcal{F}$ is a locally free sheaf on $X$, then we denote the induced vector bundle on $S$ by $F$.
\medskip

\noindent We can now explicitely describe our situation:\\
Given an arithmetic surface $X$ and an Azumaya algebra $\mathcal{A}$ over $X$.  We know by (\ref{pullaz}) that $\mathcal{A}_{\mathbb{C}}$ is an Azumaya algebra over $X_\mathbb{C}$. But $Br(X_{\mathbb{C}})=0$ by (\ref{brauer}) so we have $\mathcal{A}_{\mathbb{C}}\cong\mathcal{E}nd_{\mathcal{O}_{X_{\mathbb{C}}}}(\mathcal{E})$ for some locally free sheaf $\mathcal{E}$ on $X_{\mathbb{C}}$. This locally free sheaf $\mathcal{E}$ induces a vector bundle $E$ on $S$ which implies that the induced vector bundle of $\mathcal{A}$ on $S$ is given by $A=E\otimes E^{*}$.
\medskip

\noindent We pick once and for all a Hermitian metric $h$ on the vector bundle $E$. The dual metric construction gives us a Hermitian metric $h^{*}$ on the dual vector bundle $E^{*}$. Finally the tensor product metric of $h$ and $h^{*}$ defines a Hermitian metric $h^{\mathcal{A}}:=h\otimes h^{*}$ on $E\otimes E^{*}$. Thus we get a Hermitian locally free sheaf $\overline{\mathcal{A}}=(\mathcal{A},h^{\mathcal{A}})$. This construction leads to the following definition:
\begin{defi}
Let $X$ be an arithmetic surface. A Hermitian Azumaya algebra $\overline{\mathcal{A}}$ over $X$ is an Azumaya algebra $\mathcal{A}$ over $X$ with a Hermitian metric on the associated vector bundle on $S$ chosen as described above.
\end{defi} 
Now if $\overline{\mathcal{A}}$ is a Hermitian Azumaya algebra and $\mathcal{M}$ is a locally projective $\mathcal{A}$-module, then we know by (\ref{morita}) that $\mathcal{M}_{\mathbb{C}}=\mathcal{E}\otimes_{\mathcal{O}_{X_{\mathbb{C}}}} \mathcal{M}'$ for some locally free sheaf $\mathcal{M}'$ on $X_{\mathbb{C}}$ which induces a vector bundle $M'$ on $S$. This shows that the induced vector bundle of $\mathcal{M}$ is given by $M=E\otimes M'$.
\medskip
 
\noindent The vector bundle $E$ still comes with the Hermitian metric $h$ and we furthermore pick a Hermitian metric $h'$ on $M'$. The tensor product metric of $h$ and $h'$ yields a Hermitian metric $h^{\mathcal{M}}:=h\otimes h'$ on $M$. This defines a Hermitian locally free sheaf $\overline{\mathcal{M}}=(\mathcal{M},h^{\mathcal{M}})$, which also has the structure of a locally projective left $\mathcal{A}$-moudle. This suggests the following definition:
\begin{defi}\label{hermmod}
Let $X$ be an arithmetic surface and let $\overline{\mathcal{A}}$ be a Hermitian Azumaya algebra over $X$. A Hermitian Azumaya module $\overline{\mathcal{M}}$ is a couple $(\mathcal{M},h^{\mathcal{M}})$ where $\mathcal{M}$ is a locally projective $\mathcal{A}$-module and $h^{\mathcal{M}}$ is a Hermitian metric on the associated vector bundle on $S$ chosen as described above.
\end{defi}
Given two Hermitian Azumaya modules $\overline{\mathcal{M}}$ and $\overline{\mathcal{N}}$ we also often work with the sheaf $\mathcal{H}om_{\mathcal{A}}(\mathcal{M},\mathcal{N})$. Firstly we want to find a Hermitian metric for this sheaf.\\ 
As both modules are locally projective over $\mathcal{A}$ we see that in fact $\mathcal{H}om_{\mathcal{A}}(\mathcal{M},\mathcal{N})$ is locally free as an $\mathcal{O}_X$-module and hence we can look at the induced vector bundle on $S$. For this we remember that we have $\mathcal{A_{\mathbb{C}}}\cong \mathcal{E}nd_{\mathcal{O}_{X_{\mathbb{C}}}}(\mathcal{E})$ for some locally free sheaf $\mathcal{E}$ on $X_{\mathbb{C}}$ and $\mathcal{M}_{\mathbb{C}}\cong \mathcal{E}\otimes \mathcal{M}'$ as well as $\mathcal{N}_{\mathbb{C}}\cong \mathcal{E}\otimes \mathcal{N}'$ by Morita equivalence. This equivalence also gives an isomorphism 
\begin{center}
$\mathcal{H}om_{\mathcal{A_{\mathbb{C}}}}(\mathcal{M}_{\mathbb{C}},\mathcal{N}_{\mathbb{C}})\cong \mathcal{H}om_{\mathcal{O}_{X_\mathbb{C}}}(\mathcal{M'},\mathcal{N'})$.
\end{center}
But the vector bundle associated to the last sheaf is given by $(M')^{*}\otimes N'$.
\medskip

\noindent This bundle comes naturally equipped with the Hermitian metric $h^{(\mathcal{M},\mathcal{N})}:=(h')^{*}\otimes h''$, the one given by the Hermitian metrics $h'$ and $h''$ on the vector bundles $M'$ and $N'$. So we also have a Hermitian metric on the vector bundle associated to $\mathcal{H}om_{\mathcal{A}}(\mathcal{M},\mathcal{N})$. This construction defines the Hermitian locally free sheaf
\begin{equation}\label{hom}
\overline{\mathcal{H}om_{\mathcal{A}}(\mathcal{M},\mathcal{N})}:=(\mathcal{H}om_{\mathcal{A}}(\mathcal{M},\mathcal{N}),h^{(\mathcal{M},\mathcal{N})}).
\end{equation}

\section{Deligne pairing}
\subsection{Determinant of cohomology for pairs}
Assume $\mathcal{A}$ is an Azumaya algebra over the arithmetic surface $X$ and $\mathcal{M}$ and $\mathcal{N}$ are Azumaya modules on $X$. In this section we want to define a line bundle on $Y$, the $\mathcal{A}$-determinant of the cohomology for the pair $(\mathcal{M},\mathcal{N})$, using ideas of Borek (\cite[Section 2.3]{borek}). We will then go on and define a metric on this line bundle to get a Hermitian line bundle on $Y$.
\medskip

\noindent We will often use the usual determinant of the cohomology for a locally free sheaf $\mathcal{E}$ on $X$ which is given in this situation by: 
\begin{center}
$\lambda(\mathcal{E})=det(H^0(X,\mathcal{E}))\otimes_{\mathbb{Z}}det(H^1(X,\mathcal{E}))^{-1}$.
\end{center}
See for example (\cite[VI.1.4]{soul}) for this fact and more information on the determinant of the cohomology.
\begin{rem}
\normalfont Here $det$ is defined for any finitely generated $\mathbb{Z}$-module in the following way: write $M=F(M)\oplus T(M)$ with $F(M)$ a free $\mathbb{Z}$-module and $T(M)$ the torsion submodule of $M$. If $\#T(M)=a$ we define 
\begin{center}
$det(M)=det(F(M))\otimes_{\mathbb{Z}}\frac{1}{a}\mathbb{Z}$.
\end{center}
\end{rem}
\begin{defi}
The determinant of the cohomology of the pair $(\mathcal{M},\mathcal{N})$ is defined by
\begin{center}
$\lambda_{\mathcal{A}}(\mathcal{M},\mathcal{N}):=det(Hom_{\mathcal{A}}(\mathcal{M},\mathcal{N}))\otimes_{\mathbb{Z}}det(Ext^1_{\mathcal{A}}(\mathcal{M},\mathcal{N}))^{-1}$.
\end{center}
\end{defi}
\begin{lem}
If $X$ is an arithmetic surface, $\mathcal{A}$ is an Azumaya algebra over $X$ and $(\mathcal{M},\mathcal{N})$ is a pair of Azumaya modules on $X$, then there is the following relation between $\lambda_{\mathcal{A}}$ and $\lambda$:
\begin{center}
$\lambda_{\mathcal{A}}(\mathcal{M},\mathcal{N})=\lambda(\mathcal{H}om_{\mathcal{A}}(\mathcal{M},\mathcal{N}))$.
\end{center}
\end{lem}
\begin{proof}
We have $Hom_{\mathcal{A}}(\mathcal{M},\mathcal{N})=H^0(X,\mathcal{H}om_{\mathcal{A}}(\mathcal{M},\mathcal{N}))$. Since $\mathcal{M}$ is a locally projective $\mathcal{A}$-module we furthermore get $Ext^1_{\mathcal{A}}(\mathcal{M},\mathcal{N})=H^1(X,\mathcal{H}om_{\mathcal{A}}(\mathcal{M},\mathcal{N}))$. Now the result follows by comparing the definitions.
\end{proof}
If we look at the determinant of the cohomology $\lambda(\mathcal{E})$ of a locally free sheaf $\mathcal{E}$ and this sheaf is in fact given by a Hermitian locally free sheaf $\overline{\mathcal{E}}=(\mathcal{E},h^{\mathcal{E}})$, we can endow the determinant of the cohomology $\lambda(\mathcal{E})$ with the Quillen metric $h^{\mathcal{E}}_Q$ and get the Hermitian line bundle $\overline{\lambda(\mathcal{E})}=(\lambda(\mathcal{E}),h^{\mathcal{E}}_Q)$ on $Y$, see (\cite[VI.3]{soul}) for information and further literature about the Quillen metric.
\medskip

\noindent Now if we have a Hermitian Azumaya algebra $\overline{\mathcal{A}}$ and two Hermitian Azumaya modules $\overline{\mathcal{M}}$ and $\overline{\mathcal{N}}$, then $\mathcal{H}om_{\mathcal{A}}(\mathcal{M},\mathcal{N})$ is naturally a Hermitian locally free sheaf using the metric $h^{(\mathcal{M},\mathcal{N})}$, see (\ref{hom}). Thus we can equip the $\mathcal{A}$-determinant of the cohomology with the Quillen metric $h_Q^{(\mathcal{M},\mathcal{N})}$ to get the Hermitian line bundle $\overline{\lambda_{\mathcal{A}}(\mathcal{M},\mathcal{N})}$ on $Y$.

\subsection{A Deligne pairing for Azumaya modules}
If X is an arithmetic surface and we have two Hermitian line bundles $\overline{\mathcal{L}}$ and $\overline{\mathcal{M}}$ on $X$, then in Arakelov geometry these line bundles correspond to two Arakelov divisors on X. Now one is interested in their intersection number, which is a real number in this case. Deligne defined a line bundle $\left\langle \mathcal{L},\mathcal{M}\right\rangle_D$ on $Y$ and equipped this with a Hermitian metric using the metrics from the line bundles, see (\cite[6.3.1]{deli}). The Hermitian line bundle $\overline{\left\langle \mathcal{L},\mathcal{M}\right\rangle_D}$ has the property that its Arakelov degree is equal to the intersection number of the associated divisors, see (\cite[1.2.(15)]{soul3}). In this section we want to generalize this line bundle to Hermitian Azumaya modules.
\begin{defi}
Given a Hermitian Azumaya algebra $\overline{\mathcal{A}}$ over the arithmetic surface $X$. If $(\overline{\mathcal{M}},\overline{\mathcal{N}})$ is a pair of Hermitian Azumaya modules, then we define the $\mathcal{A}$-Deligne pairing of the pair as the Hermitian line bundle on $Y$ given by:
\begin{center}
$\overline{\left\langle \mathcal{M},\mathcal{N}\right\rangle_{\mathcal{A}}}=\overline{\lambda_{\mathcal{A}}(\mathcal{M},\mathcal{N})}\otimes_{\mathbb{Z}} \overline{\lambda_{\mathcal{A}}(\mathcal{M},\mathcal{A})}^{(-1)}\otimes_{\mathbb{Z}} \overline{\lambda_{\mathcal{A}}(\mathcal{A},\mathcal{N})}^{(-1)}\otimes_{\mathbb{Z}} \overline{\lambda_{\mathcal{A}}(\mathcal{A},\mathcal{A})}$.
\end{center}
\end{defi}
A first question is, if we can find pairs of Azumaya modules, such that the pairng is trivial. The following proposition gives one easy method to find such pairs. 
\begin{prop}
Given a Hermitian Azumaya algebra $\overline{\mathcal{A}}$ over the arithmetic surface $X$. If one of the elements in the pair $(\overline{\mathcal{M}},\overline{\mathcal{N}})$ of Hermitian Azumaya modules is the Hermitian Azumaya algebra itself, then we have:
\begin{center}
$\overline{\left\langle \mathcal{M},\mathcal{N}\right\rangle_{\mathcal{A}}}=\overline{\mathcal{O}_Y}$. 
\end{center}
Here $\overline{\mathcal{O}_Y}$ is the structure sheaf of $Y$ equipped with the trivial metric.
\end{prop}
\begin{proof}
If one element of the pair $(\overline{\mathcal{M}},\overline{\mathcal{N}})$ is equal to $\overline{\mathcal{A}}$, then we see that the four Hermitian line bundles in $\overline{\left\langle \mathcal{M},\mathcal{N}\right\rangle_{\mathcal{A}}}$ can be grouped into two pairs which chancel each other by definition.
\end{proof}
The main observation of this note is that our definition of $\overline{\left\langle \mathcal{M},\mathcal{N}\right\rangle_{\mathcal{A}}}$ agrees up to dualization with the usual Deligne pairing for line bundles if we restrict to Azumaya line bundles, and if we pick the trivial Hermitian Azumaya algebra $\overline{\mathcal{A}}=\overline{\mathcal{O}_X}$.
\medskip

\noindent To prove the mentioned result we need the following different description of the Deligne pairing, see \cite[A.2. Proposition A.1.]{weng} or \cite[1.12. Theorem]{voron}.
\begin{lem}\label{diff}
Assume $X$ is an arithmetic surface and $\overline{\mathcal{L}}$ and $\overline{\mathcal{M}}$ are two Hermitian line bundles on $X$. Define a line bundle $\mathcal{P}$ on $Y$ by
\begin{center}
$\mathcal{P}:=\lambda(\mathcal{L}\otimes\mathcal{M})\otimes_{\mathbb{Z}}\lambda(\mathcal{L})^{-1}\otimes_{\mathbb{Z}}\lambda(\mathcal{M})^{-1}\otimes_{\mathbb{Z}}\lambda(\mathcal{O}_X)$.
\end{center}
If we equip each determinant of the cohomology with the Quillen metric we get a Hermitian line bundle $\overline{\mathcal{P}}$ on $Y$ and there is an isometry
\begin{center}
$\overline{\left\langle \mathcal{L},\mathcal{M}\right\rangle_D}\cong \overline{\mathcal{P}}$.
\end{center}
\end{lem}
\begin{thm}
Assume $X$ is an arithmetic surface, $\overline{\mathcal{A}}$ is a Hermitian Azumaya algebra over $X$ and $(\overline{\mathcal{M}},\overline{\mathcal{N}})$ is a pair of Hermitian Azumaya line bundles on $X$. If we pick the trivial Hermitian Azumaya algebra $\overline{\mathcal{A}}=\overline{\mathcal{O}_X}$, then we have an isometry:
\begin{center}
$\overline{\left\langle \mathcal{M},\mathcal{N}\right\rangle_{\mathcal{A}}}\cong \overline{\left\langle \mathcal{M},\mathcal{N}\right\rangle_D}^{(-1)}$
\end{center}
\end{thm}
\begin{proof}
If $\mathcal{A}=\mathcal{O}_X$, then we see that the Azumaya line bundles $\mathcal{M}$ and $\mathcal{N}$ are nothing more but usual line bundles on $X$. Since $\overline{\mathcal{O}_X}$ carries the trivial metric we see that the metrics constructed on the Hermitian Azumaya modules in (\ref{hermmod}) are just normal Hermitian metrics, that is, there is no tensor product metric involved.\\ 
Furthermore we see that $\mathcal{H}om_\mathcal{A}(\mathcal{M},\mathcal{N})=\mathcal{M}^{-1}\otimes_{\mathcal{O}_X}\mathcal{N}$. This implies 
\begin{center}
$\lambda_{\mathcal{A}}(\mathcal{M},\mathcal{N})=\lambda(\mathcal{M}^{-1}\otimes_{\mathcal{O}_X}\mathcal{N})$. 
\end{center}
But then we recognize that $\overline{\left\langle \mathcal{M},\mathcal{N}\right\rangle_{\mathcal{A}}}=\overline{\mathcal{P}}$ where $\overline{\mathcal{P}}$ is the Hermitian line bundle on $Y$ described in (\ref{diff}) for the Hermitian line bundles $\mathcal{M}^{-1}$ and $\mathcal{N}$ on $X$.\\ 
This shows that there is an isometry 
\begin{center}
$\overline{\left\langle \mathcal{M},\mathcal{N}\right\rangle_{\mathcal{A}}}\cong \overline{\left\langle \mathcal{M}^{-1},\mathcal{N}\right\rangle_D}$.
\end{center}
The usual Deligne pairing is bilinear, see (\cite[6.2]{deli}). So we have for example:
\begin{center}
$\left\langle \mathcal{L}\otimes_{\mathcal{O}_X} \mathcal{L}',\mathcal{M}\right\rangle_D\cong\left\langle \mathcal{L},\mathcal{M}\right\rangle_D\otimes_{\mathbb{Z}} \left\langle \mathcal{L}',\mathcal{M}\right\rangle_D$.
\end{center}
This isomorphism of line bundles becomes an isometry if we equip each line bundle with the Deligne metric, see (\cite[6.5]{deli}). Using this we immediately see that there is an isometry
\begin{center}
$\overline{\left\langle \mathcal{O}_X,\mathcal{M}\right\rangle_D}\cong \overline{\mathcal{O}_Y}$.
\end{center}
Using these facts for $\mathcal{L}'=\mathcal{L}^{-1}$ we see that there is an isometry
\begin{center}
$\overline{\left\langle \mathcal{L}^{-1},\mathcal{M}\right\rangle_D}\cong \overline{\left\langle \mathcal{L},\mathcal{M}\right\rangle_D}^{(-1)}$.
\end{center}
Putting everything together we get an isometry
\begin{center}
$\overline{\left\langle \mathcal{M},\mathcal{N}\right\rangle_{\mathcal{A}}}\cong \overline{\left\langle \mathcal{M}^{-1},\mathcal{N}\right\rangle_D}\cong \overline{\left\langle \mathcal{M},\mathcal{N}\right\rangle_D}^{(-1)}$.
\end{center}
\end{proof}
\medskip

\noindent The questions remains, if we can somehow compare this $\mathcal{A}$-Deligne pairing with the one given by Borek in (\cite[Section 2.3]{borek}). If $\overline{\mathcal{E}}$ is a Hermitian locally free sheaf, then he replaces the Hermitian metric on the associated vector bundle $E$ on the Riemann surface $S$ by an automorphism $\beta$ of the sheaf $\mathcal{E}_{\mathbb{R}}$,  (\cite[Definition 2.2.4]{borek}). The Quillen metric is replaced by the determinant of the action of $\beta$ on the $Ext$-groups, (\cite[Section 2.3.2]{borek}). Is there any relationship between these two approaches?
\medskip

\noindent Furthermore one can ask the following questions:
\begin{itemize} 
\item Is there a Riemann-Roch type formula for the $\mathcal{A}$-Deligne pairing?
\item How does the $\mathcal{A}$-determinant of the cohomology behave with respect to Serre duality? Do we have to adjust the chosen metrics to get better results?
\item Can we find arithmetic $\mathcal{A}$-Chern classes to get a similar formula for the first arithmetic Chern class of the $\mathcal{A}$-Deligne pairing, like for the usual Deligne pairing, see (\cite[Theorem 4.10.1.]{soul2}): 
\begin{center}
$\widehat{c}_1(\overline{\left\langle \mathcal{L},\mathcal{M}\right\rangle_D})=\pi_{*}(\widehat{c}_1(\overline{\mathcal{L}}).\widehat{c}_1(\overline{\mathcal{L}}))$.
\end{center}  
\end{itemize}

\addcontentsline{toc}{section}{References}
\bibliography{Artikel}
\bibliographystyle{alphaurl}

\end{document}